\documentclass[11pt]{amsart}
\usepackage{geometry}                
\usepackage{graphicx}
\usepackage{amssymb}
\usepackage{epstopdf}
\newtheorem{theorem}{Theorem}[section]

\newtheorem{corollary}{Corollary}[section]

\newtheorem{conjecture}{Conjecture}[section]
\newtheorem{definition}{Definition}[section]

\numberwithin{equation}{section}

\def\Z{\mathbb Z}

\def\R{\mathbb R}

\def\N{\mathbb N}

\def\d{\partial}
\def\a{\alpha}

\def\g{\gamma}
\def\s{\sigma}

\def\D{\Delta}

\def\om{\omega}

\title{Gromov's Amenable Localization and Geodesic Flows}
\author{Gabriel Katz}

\address{MIT, Department of Mathematics, 77 Massachusetts Ave., Cambridge, MA 02139, U.S.A.}
\email{gabkatz@gmail.com}

\begin{document}
\maketitle

\begin{abstract} 
Let $M$ be a compact smooth Riemannian $n$-manifold with boundary.  We combine Gromov's amenable localization technique with the Poincar\'{e} duality to study the {\sf traversally generic} geodesic flows on $SM$, the space of the spherical tangent bundle. Such flows generate  stratifications of $SM$, governed by rich universal combinatorics. The stratification reflects the ways in which the flow trajectories are tangent to the boundary $\d(SM)$. Specifically, we get lower estimates of the numbers of connected components of these flow-generated strata of any given codimension $k$ in terms of the normed homology $H_k(M; \R)$ and $H_k(DM; \R)$, where $DM = M\cup_{\d M} M$ denotes the double of $M$. The norms here are the {\sf simplicial semi-norms} in homology. The more complex the metric on $M$ is, the more numerous the strata of $SM$ and $S(DM)$ are. 
It turns out that the normed homology spaces form obstructions to the existence of globally $k$-{\sf convex} traversally generic metrics on $M$. 
We also prove that knowing the geodesic scattering map on $M$ makes it possible to reconstruct the stratified topological type of the space of geodesics, as well as the amenably localized Poincar\'{e} duality operators on $SM$.
\end{abstract} 

\section{Introduction}

This paper is an extension of \cite{AK} and especially of \cite{K3}. As in the latter articles, it draws its inspiration from the papers \cite{Gr}, \cite{Gr1} of Gromov, where the machinery of {\sf amenable localization} has been developed. Here we combine Gromov's amenable localization with the Poincar\'{e} duality operators (this combination was introduced in \cite{K3}) to study the {\sf traversally generic geodesic flows} on compact connected smooth Riemannian $n$-manifolds $M$ with boundary  (see Definition \ref{def2.3} and formula (\ref{eq2.4}) for the  definition of a traversally generic vector field.). For such flows, the metric on $M$ is {\sf non-trapping} (see Definition \ref{def3.1}). \smallskip

The application of the techniques from \cite{K3} to geodesic flows is relatively straightforward; however, the idea that such applications may be fruitful is a novelty.   
These applications reveal a new phenomenon: the {\sf simplicial norms} (see Definition \ref{def_simplicial}) of homology classes of a Riemannian manifold $M$ impose restrictions on the {\sf tangency patterns} of geodesic curves to its boundary $\d M$. A tangency pattern $\om =_{\mathsf{def}} (j_1, j_2, \dots , j_i,  \dots )$ is just an ordered finite string of natural numbers. They encode the degree of tangency of a geodesic curve $\g \subset M$ to the boundary $\d M$.

The fundamental groups of the boundary components of $\d M$ are assumed to be {\sf amenable}, while the non-trivial results arise only for manifolds $M$ with non-amenable fundamental group.\smallskip

For a given Riemannian metric $g$ on a compact manifold $M$, we denote by $v^g$ the geodesic vector field on $SM$, the space of the $g$-unitary spherical tangent bundle of $M$. 

In \cite{K5}, we have studied metrics $g$ on $M$ such that the geodesic field $v^g$ on $SM$ is of the {\sf gradient type}; that is, there exists a smooth Lyapunov function $F: SM \to \R$ with the property $dF(v^g) > 0$. By Theorem 2.1  and Corollary 2.3 from \cite{K5}, such metrics $g$ coincide with the class of non-trapping metrics on $M$. They form an open nonempty set $\mathcal G(M)$ in the space $\mathcal M(M)$ of all Riemannian metrics on $M$.\smallskip

In \cite{K5}, Definition 2.3, we also introduced the notion of {\sf boundary generic metrics} (see Definitions \ref{def2.2} and \ref{def3.1} below). For them, the boundary $\d M$ is ``generically curved" and the geodesics $\g \subset M$ do not have high order of tangency to $\d M$ in comparison to its dimension. We denote by $\mathcal B^\dagger(M)$ the space of boundary generic metrics, and by $\mathcal G^\dagger(M)$ the space of boundary generic {\it and} non-trapping metrics. If each component of $\d M$ is either strictly convex or concave in $g$, then the metric $g$ is boundary generic.  

Boundary generic metrics form an open set in the space $\mathcal M(M)$ of all metrics. Conjecturally, $\mathcal B^\dagger(M)$ is dense in $\mathcal M(M)$, and 
$\mathcal G^\dagger(M)$ is dense in the space $\mathcal G(M)$. \smallskip

Finally, we consider a subspace $\mathcal G^\ddagger(M) \subset \mathcal G^\dagger(M)$, formed by the {\sf traversally generic metrics} on $M$. Their Definition \ref{def2.3} is more involved. 
In short, a metric $g$ is traversally generic if the geodesic vector field $v^g$ on $SM$ is traversally generic with respect to the boundary $\d(SM)$ in the sense of \cite{K2}, Definition 3.2. 
Conjecturally, $\mathcal G^\ddagger(M)$ is also dense in the space $\mathcal G(M)$. It is proven to be open in $\mathcal M(M)$ (see \cite{K2}).
\smallskip

For a given metric $g \in \mathcal G^\dagger(M)$, let $\mathcal T(v^g)$ denote the space of trajectories of the $v^g$-generated flow on $SM$.  We call $\mathcal T(v^g)$  {\sf the space of geodesics} on $M$. We denote by $\Gamma: SM \to \mathcal T(v^g)$ the obvious map that takes each point $x \in SM$ to the $v^g$-trajectory through $x$. In general, $\mathcal T(v^g)$ is not a manifold, but, for $g \in \mathcal G^\ddagger(M)$, a compact $CW$-complex (\cite{K5}). \smallskip

For a smooth {\sf traversally generic vector field} $v$ on a compact $(d+1)$-manifold $X$ with boundary, the trajectory space $\mathcal T(v)$  acquires a stratification $\{\mathcal T(v, \om)\}_\om$, labeled by the combinatorial patterns of tangency $\om$ that belong to an {\sf universal poset} $\mathbf\Omega^\bullet_{'\langle d]}$  (see \cite{K2} and Section 2). This poset depends only on $d$. 

Similarly, for a traversally generic vector field $v^g$ on a compact $(2n-1)$-manifold $SM$, the $(2n-2)$-dimensional space of geodesics $\mathcal T(v^g)$  acquires a stratification $\{\mathcal T(v^g, \om)\}_\om$ by the combinatorial patterns of tangency $\om$ that belong to the universal poset $\mathbf\Omega^\bullet_{'\langle 2n-2]}$. 

The more numerous the connected components of these stratifications are, the more \emph{complex} the $v^g$-flow is, and thus the more complex (in relation to $\d M$) the metric $g$ on $M$ is. 

So, our goal here is to find \emph{lower bounds} of the numbers such connected components in terms of the topology of $SM$ or $M$. The key tool, enabling such estimates, has been developed in \cite{K3}  for arbitrary traversally generic flows. We call it ``{\sf Gromov's amenable localization of the Poincar\'{e} duality}". 
\smallskip

The $\mathbf\Omega^\bullet_{'\langle 2n-2]}$-stratification of the geodesic space $\mathcal T(v^g)$ generates the stratification \hfill\break $$\big\{SM(v^g, \om) =_{\mathsf{def}}  \Gamma^{-1}(\mathcal T(v^g, \om))\big\}_{\om \in \mathbf\Omega^\bullet_{'\langle 2n-2]}}$$ of $SM$ and the stratification $$\big\{\d(SM)(v^g,\om) =_{\mathsf{def}} SM(v^g, \om) \cap \d(SM)\big\}_\om$$ of its boundary $\d(SM)$. These stratifications can be refined by considering the {\sf connected components} of the sets $\{\d(SM)(v^g, \om),\; SM^\circ(v^g, \om)\}_\om$. Here $$SM^\circ(v^g, \om) =_{\mathsf{def}} SM(v^g, \om) \cap \textup{int}(SM).$$ 

We construct an auxiliary closed manifold, the {\sf double} $D(SM) =_{\mathsf{def}} SM \cup_{\d (SM)} SM$ of $SM$. The double comes equipped with an involution $\tau$ so that $(D(SM))^\tau = \d (SM)$ and the orbit-space $D(SM)/\{\tau\}$ is homeomorphic to $SM$. We stratify $D(SM)$ by the connected components of the following sets: $$\big\{\d(SM)(v^g, \om),\;\, SM^\circ(v^g, \om),\; \,\tau(SM^\circ(v^g, \om))\big \}_{\om \in \mathbf\Omega^\bullet_{'\langle 2n-2]}}$$

All these $v^g$-induced stratifications of $\mathcal T(v)$, $SM$, and $D(SM)$ are the foci of our investigation here. \smallskip

Let $SM^\circ_{-j}(g)$ denote the union of strata of codimension $j$ in $SM^\circ$. Similarly, let $D(SM)_{-j}(g)$ denote the union of codimension $j$ strata in $D(SM)$.\smallskip 

Let $\mathcal D$ stand for the Poincar\'{e} Duality operator (over the coefficient rings $\R$ or $\Z$) on an oriented manifold with boundary. For each tranversally generic metric $g$ on $M$, we introduce two {\sf localized} Poincar\'{e} Duality linear operators:
\[
\mathcal L_{j}(g; \R):\, H_{j}(D(SM); \R) \stackrel{\approx \mathcal D}{\longrightarrow}  H^{2n -1 -j}(D(SM); \R) \stackrel{i^\ast_{\mathsf{loc}}}{\longrightarrow}  
H^{2n -1 -j}(D(SM)_{-j}(g); \R),
\]
\[
\mathcal M_{j}(g; \R):\, H_{j}(SM; \R)\stackrel{\approx \mathcal D}{\longrightarrow}  H^{2n -1 -j}(SM, \d(SM); \R) 
\]
\[
 \stackrel{i^\ast_{\mathsf{loc}}}{\longrightarrow}  H^{2n -1 -j}(SM_{-j}(g), SM_{-j}(g) \cap \d(SM); \R),  
\]
where $i^\ast_{\mathsf{loc}}$ are the natural homomorphisms in the cohomology, induced by the corresponding embeddings of spaces.\smallskip

The source spaces of these operators come naturally equipped with the {\sf Gromov simplicial semi-norms} $\| \sim \|_\mathbf \D$ (\cite{Gr}). Let us recall their definition. 

\begin{definition}\label{def_simplicial}
Let $X$ be a topological space, and let $h \in H_j(X; \R)$ be a real homology class. 
We consider singular real cycles $c = \sum_{i=1}^\infty r_i \s_i$ that represent $h$, where the coefficients $r_i \in \R$ and $\s_i: \D^j \to X$ are singular $j$-simplexes, such that  each compact $K\subset X$ intersects only with the images $\{\s_i(\D^k)\}_i$ of finitely many singular simplexes $\s_i$. \smallskip

Put $\|c\|_{\ell_1} =_{\mathsf{def}}\; \sum_i |r_i|$.\footnote{For a non-compact $X$, $\|c\|_{\ell_1}$ may be infinite.} Then $$\|h\|_\mathbf\D =_{\mathsf{def}}\; \inf_{\{c\}} \{\|c\|_{\ell_1}\}.$$
 A similar definition is available for relative classes $h \in H_j(X, Y; \R)$, where $Y$ is a subspace of $X$. \hfill $\diamondsuit$
\end{definition}

We denote by $\tilde{\mathsf B}^{\mathbf \D}_{j}(X)$ the unit balls in these semi-norms. Thus $\tilde{\mathsf B}^{\mathbf \D}_{j}(X) \subset H_j(X; \R)$ is a convex set, possibly  non-compact. 

We will see that the $g$-dependent target spaces of the operators $\mathcal L_{j}(g; \R)$ and $\mathcal M_{j}(g; \R)$ may be also equipped with some norms $\| \, \sim \, \|^\bullet_{\mathbf\mho}$ (see the text that follows the proof of Theorem \ref{th11.15}). Their definition depends only on the codimension $j$ connected components of the $\mathbf\Omega^\bullet_{'\langle 2n-2]}$-stratifications of the  spaces $D(SM)$ and $SM$, respectively. For each $j$, the unit balls, $\diamondsuit^{(j)}_{\mathbf\mho}(SM, g)$ and $\diamondsuit^{(j)}_{\mathbf\mho}(D(SM), g)$, in these norms $\| \, \sim \, \|^\bullet_{\mathbf\mho}$ are compact convex polyhedra. They depend on the geodesic flow, and thus on the metric $g$. In fact, the balls are linear projections of some perfect polyhedra in the appropriate vector spaces (these perfect polyhedra are the duals of the hypercubes ``$\Box$"; thus the notation ``$\diamondsuit$"). \smallskip

Theorem \ref{th11.16} describes the images, under the localized Poincar\'{e} Duality operators $\mathcal L_{j}(g; \R)$ and $\mathcal M_{j}(g; \R)$, of the unitary spheres $\d\tilde{\mathsf B}^{\mathbf \D}_{j}(\sim)$ in Gromov's semi-norms. 


We stress that all these results require that the fundamental groups $\pi_1(M)$ and $\pi_1(DM)$ are quite big ({\sf non-amenable}); for the amenable fundamental groups, all our results are vacuous!  
\smallskip

For a given homology group $H_j(\sim; \R)$, we form its quotient space 
$$H^{\mathbf \D}_j(\sim; \R) =_{\mathsf{def}}\; H_j(\sim; \R)\big/ H_j^{\{\|\sim\|_\D = 0\}}(\sim; \R),$$ where the vector subspace $H_j^{\{\|\sim\|_\D = 0\}}(\sim; \R)$ is spanned by all the homology classes whose simplicial semi-norm vanishes. So there is an obvious epimorphism $H_j(\sim; \R) \to H^{\mathbf \D}_j(\sim; \R)$, which converts the semi-norm $\|\,\sim \, \|_{\mathbf \D}$ on $H_j(\sim; \R)$ into a {\it norm} $\|\,\sim \, \|_{\mathbf \D}^\bullet$ on $H^{\mathbf \D}_j(\sim; \R)$. We call the {\it normed} vector space $H^{\mathbf \D}_j(\sim; \R)$ ``{\sf the reduced $j$-homology}".
\smallskip
\smallskip

Theorem \ref{th11.16} implies that the kernels of the localized Poincar\'{e} Duality operators $\mathcal L_{j}(g; \R)$ and $\mathcal M_{j}(g; \R)$ are contained in the  spaces $H_j^{\{\|\sim\|_\D = 0\}}(D(SM); \R)$ and $H_j^{\{\|\sim\|_\D = 0\}}(SM; \R)$, respectively. 

This observation leads to our main result, Theorem \ref{th11.17}. It claims that, for $j \in [1, n]$, the ranks of the reduced homology groups $H^{\mathbf \D}_{j}(D(SM); \R) \approx H^{\mathbf \D}_{j}(DM; \R)$ give a \emph{lower bound} for the number of connected components of the codimension $j$ strata $$\{\d(SM)(v^g, \om),\; SM^\circ(v^g,\om),\  \tau(SM^\circ(v^g, \om)) \}_{\om}.$$ 

Similarly, for $j \in [1, n]$, the ranks of the reduced homology groups $H^{\mathbf \D}_{j}(SM; \R) \approx H^{\mathbf \D}_{j}(M; \R)$ give a lower bound for the number of connected components of the codimension $j$ strata $\{SM^\circ(v^g, \om) \cap \textup{int}(SM)\}_{\om}$. 

Both claims may be regarded as vague analogues of the Morse inequalities for the spaces of geodesics.\smallskip

The reader may glance at reveling Examples 3.1 and 3.2 to get a better feel for the claims of our main results, Theorem \ref{th11.16} and Theorem \ref{th11.17}. \smallskip

We say that a metric $g \in \mathcal G^\ddagger(M)$ is {\sf globally $j$-convex} if the $v^g$-induced stratification of $\mathcal T(v^g)$ has no strata of codimension $\geq j$ (compare this with Definition \ref{def2.3}). For example, if $\d M$ is strictly convex or concave in $g$,  then $g$ is globally $3$-convex.
By Corollaries \ref{cor11.A} and \ref{cor11.B}, the non-triviality of the groups $H^{\mathbf \D}_{j}(SM; \R)$ and $H^{\mathbf \D}_{j}(SM; \R)$ constitutes an \emph{obstruction} to the existence of a global $j$-convex metric $g \in \mathcal G^\ddagger(M)$. \smallskip

For $g \in \mathcal G^\ddagger(M)$, we are also investigating the connection of the localized Poincar\'{e} Duality operators with {\sf the inverse geodesic scattering problem}.
\smallskip 

The {\sf scattering map} $C_{v^g}: \d_1^+(SM) \to \d_1^-(SM)$, generated by the $v^g$-flow, takes a domain $\d_1^+(SM)$ in the boundary $\d(SM)$ to the complementary domain $\d_1^-(SM) \subset \d(SM)$, both domains being diffeomorphic and $g$-independent. 

Let us outline the construction of $C_{v^g}$. For any point $x \in \d M$ and any unit tangent vector $w \in T_xM$ that points inside of $M$ or is tangent to $\d M$, we consider the geodesic $\g \subset M$. We take the next along $\g$ point $x' \in \g \cap \d M$ and register the unit tangent to $\g$ vector $w' \in T_{x'}M$. By definition, $C_{v^g}(x, w) = (x', w')$. In fact, $\d_1^+(SM)$ is diffeomorphic to $\d_1^-(SM)$, the smooth topological types of both domains being $g$-independent. With a few exceptions, fundamentally, $C_{v^g}$ is a \emph{discontinuous} map! 

It turns out that $C_{v^g}$ allows for a reconstruction of the 
${\mathbf\Omega}^\bullet_{'\langle 2n-2]}$-\emph{stratified} topological type of $SM$ (see Theorem 3.3 from \cite{K5}). This reconstruction is an instance of a more general phenomenon, which we call ``The Holographic Principle" (see Theorem 3.1 from \cite{K4} and \cite{K6}). By applying the Holographic Principle to the geodesic flows, in Theorem \ref{th11.16a} below, we prove  that it is possible to reconstruct  the localized Poincar\'{e} Duality operators from the geodesic scattering map $C_{v^g}$.  
\section{Basics of Traversally Generic Vector Fields}

In an attempt to make this text more self-contained, we start with presenting few basic definitions and facts related to the traversing and traversally generic vector fields, as they appear in \cite{K1} - \cite{K5} and \cite{K7}. 
\smallskip

Let $X$ be a compact connected smooth $(n+1)$-dimensional manifold with boundary. A vector field $v$ is called {\sf traversing} if each $v$-trajectory is ether a closed interval with both ends residing in $\d X$, or a singleton also residing in $\d X$ (see \cite{K1} for the details). In particular, a traversing vector field does not vanish in $X$. In fact, $v$ is traversing if and only if $v \neq 0$ and $v$ is of the gradient type (\cite{K1}, Corollary 4.1).
\smallskip 

We denote by $\mathcal V_{\mathsf{trav}}(X)$ the space of traversing fields on $X$. \smallskip

For traversing fields $v$, the trajectory space $\mathcal T(v)$ is homology equivalent  to $X$ \cite{K1}. 
\smallskip

In this paper, we consider an important subclass of traversing fields which we call {\sf traversally generic} (see formula (\ref{eq2.4}) below and Definition 3.2 from \cite{K2}).

For a traversally  generic field $v$, the trajectory space $\mathcal T(v)$ is stratified by closed subspaces, labeled by the elements $\om$ of an \emph{universal} poset ${\mathbf\Omega}^\bullet_{'\langle n]}$.   It depends only on $\dim(X) = n+1$. The partial order ``$\succ_\bullet$" in ${\mathbf\Omega}^\bullet_{'\langle n]}$ mimics the bifurcations of real roots of real polynomials of degree $2n +2$ (see Section 2, for the definition and properties of ${\mathbf\Omega}^\bullet_{'\langle n]}$). The elements $\om \in \mathbf\Omega^\bullet_{'\langle n]}$ correspond to combinatorial patterns that describe the way in which $v$-trajectories $\g \subset X$ intersect the boundary $\d X$. Each intersection point $a \in \g \cap \d X$ acquires a well-defined {\sf multiplicity} $m(a)$, a natural number that reflects \emph{the order of tangency} of $\g$ to $\d X$ at $a$ (see \cite{K1} and Definition 2.1 for the expanded definition of $m(a)$). So $\g \cap \d X$ can be viewed as a \emph{divisor} $D_\g$ on the trajectory $\g$, an ordered set of points in $\g$ together with their multiplicities. Then $\om$ is just the ordered sequence of multiplicities $\{m(a)\}_{a \in \g \cap \d X }$, the order being prescribed by $v$. \smallskip

The support of the divisor $D_\g$ is either a singleton $a$, in which case $m(a) \equiv 0 \; \mod \, 2$, or the minimum and maximum points of $\sup D_\g$ have \emph{odd} multiplicities, and the rest of the points have \emph{even} multiplicities. \smallskip

Let 
\begin{eqnarray}\label{eq2.1}
m(\g) =_{\mathsf{def}} \sum_{a \in \g\, \cap \, \d X }\; m(a)\quad \text{and} \quad m'(\g) =_{\mathsf{def}} \sum_{a \in \g \, \cap \, \d X }\; (m(a) -1).
\end{eqnarray} 
Similarly, for $\om =_{\mathsf{def}} (j_1, j_2, \dots , j_i,  \dots )$, where $j_i \in \N$, we introduce the {\sf norm} and the {\sf reduced norm} of $\om$ by the formulas: 

\begin{eqnarray}\label{eq2.2}
|\om| =_{\mathsf{def}} \sum_i\; j_i \quad \text{and} \quad |\om|'  =_{\mathsf{def}} \sum_i\; (j_i -1).
\end{eqnarray}
\smallskip

Let $\d_jX =_{\mathsf{def}} \d_jX(v)$ denote the locus of points $a \in \d X$ such that the multiplicity of the $v$-trajectory $\g_a$ through $a$ at $a$ is greater than or equal to $j$. 

We may embed the compact manifold $X$ into an open manifold $\hat X$ of the same dimension, so that $v$ extends smoothly to a non-vanishing gradient-like vector field $\hat v$ in $\hat X$. We treat the extension $(\hat X, \hat v)$ as ``a germ at $(X, v)$". \smallskip

Now, the locus $\d_jX =_{\mathsf{def}} \d_jX(v)$ has a description in terms of an auxiliary function $z: \hat X \to \R$ that satisfies the following three properties:
\begin{eqnarray}\label{eq2.3}
\end{eqnarray}

\begin{itemize}
\item $0$ is a regular value of $z$,   
\item $z^{-1}(0) = \d X$, and 
\item $z^{-1}((-\infty, 0]) = X$. 
\end{itemize}

In terms of $z$, the locus $\d_jX =_{\mathsf{def}} \d_jX(v)$ is defined by the equations: 
$$\big\{z =0,\; \mathcal L_vz = 0,\; \ldots, \;  \mathcal L_v^{(j-1)}z = 0\big\},$$
where $\mathcal L_v^{(k)}$ stands for the $k$-th iteration of the Lie derivative operator $\mathcal L_v$ in the direction of $v$ (see \cite{K2}). 
The pure stratum $\d_jX^\circ \subset \d_jX$ is defined by the additional constraint  $\mathcal L_v^{(j)}z \neq 0$. The locus $\d_jX$ is the union of two loci: $(1)$ $\d_j^+X$, defined by the constraint  $\mathcal L_v^{(j)}z \geq  0$, and $(2)$ $\d_j^-X$, defined by the constraint  $\mathcal L_v^{(j)}z \leq  0$. The two loci, $\d_j^+X$ and $\d_j^-X$, share a common boundary $\d_{j+1}X$.
\smallskip

\begin{definition}\label{def2.1}
The {\sf multiplicity} $m(a)$, where $a \in \d X$, is the index $j$ such that $a \in \d_jX^\circ$. 

\hfill $\diamondsuit$
\end{definition}

\begin{definition}\label{def2.2}
The vector field $v$ on $X$ is called {\sf boundary generic}, if for all $j$ and each point $a \in \d_jX^\circ$,  there exists a neighborhood $V_a \subset \hat X$ of $a$ and some local coordinates $(u, \vec x, \vec y): V_a \to \R\times \R^{j-1} \times \R^{n-j}$ so that $\d X$ is given by the polynomial equation
\begin{eqnarray}\label{eq2.3a}
 \wp (u, \vec x) =_{\mathsf{def}} u^{j} + \sum_{l = 0}^{j-2} x_l\, u^l  = 0
 \end{eqnarray}
of  degree $j = m(a)$ in $u$, while $X$ by the polynomial inequality $\pm \wp (u, \vec x) \leq 0$. Each $v$-trajectory in $V$ is obtained by freezing the $\vec x, \vec y$ coordinates.
\hfill $\diamondsuit$
\end{definition}


The characteristic property of {\sf traversally generic} fields is that they admit special flow-adjusted coordinate systems, in which the boundary is given by  quite special polynomial equations (see formula (\ref{eq2.4})), and the trajectories are parallel to one of the preferred coordinate axis (see  \cite{K2}, Lemma 3.4). For a traversally generic $v$ on a $(n+1)$-dimensional $X$, the vicinity $U \subset \hat X$ of each $v$-trajectory $\g$ of the combinatorial type $\om = (j_1, j_2, \dots , j_i,  \dots )$ has a special coordinate system $$(u, \vec x, \vec y): U \to \R\times \R^{|\om|'} \times \R^{n-|\om|'}.$$  In these coordinates, by Lemma 3.4  and formula $(3.17)$ from \cite{K2}, the boundary $\d X$ is given  by the polynomial equation: 
\begin{eqnarray}\label{eq2.4}
 \wp (u, \vec x) =_{\mathsf{def}} \prod_{i=1}^{|\om| - |\om|'} \Big[(u-i)^{j_i} + \sum_{l = 0}^{j_i-2} x_{i, l}\,(u -i)^l \Big] = 0
 \end{eqnarray}
of an even degree $|\om|$ in $u$. Here  $i \in \Z_+$ runs over the distinct roots of  $\wp (u, \vec 0)$, and the vector $\vec x =_{\mathsf{def}} \{ x_{i, l}\}_{i,l}$ is sufficiently small.   
At the same time, $X$ is given by the polynomial inequality $\{\wp(u, \vec x) \leq 0\}$.  Each $v$-trajectory in $U$ is produced by freezing all the coordinates $\vec x, \vec y$, while letting $u$ to be free. 

Here we treat formula (\ref{eq2.4}) as {\it the working definition} of a traversally generic vector field; for a more conceptual definition see \cite{K2}.
\smallskip

We denote by $\mathcal V^\ddagger(X)$ the space of traversally  generic fields on $X$.  In fact,  $\mathcal V^\ddagger(X)$ is an \emph{open} and \emph{dense} (in the $C^\infty$-topology) subspace of $\mathcal V_{\mathsf{trav}}(X)$ (see \cite{K2}, Theorem 3.5).\smallskip 

We denote by $X(v, \om)$ the union of $v$-trajectories whose divisors are of a given combinatorial type $\om \in \mathbf\Omega^\bullet_{'\langle n]}$. Its closure $\bigcup_{\om' \preceq_\bullet \om}\; X(v, \om')$ is denoted by $X(v, \om_{\succeq_\bullet})$.
\smallskip

Each \emph{pure} stratum $\mathcal T(v, \om) \subset \mathcal T(v)$ is an open smooth manifold of dimension $n-|\om|'$ and, as such, has a ``conventional" tangent bundle. \smallskip

\begin{definition}\label{def2.3}
We say that a traversing field $v$ on $X$ is {\sf globally $k$-convex} if $m'(\g) < k$ for any $v$-trajectory $\g$. In other words, all strata $\mathcal T(v, \om)$ of codimension $\geq k$ are empty.  
\hfill $\diamondsuit$
\end{definition}

\section{The localized Poincar\'{e} Duality for geodesic flows}

We are now in position to apply Gromov's amenable localization to the study of a certain class of Riemmanian metrics (called ``traversally generic") on compact connected smooth manifolds $M$ with boundary. We will focus on the tangency patterns, exhibited by the geodesic curves with respect to the boundary $\d M$. We follow closely the general arguments  in \cite{K3}, as they apply to geodesic flows for non-trapping metrics.
\smallskip

Consider the unit spherical fibration $SM \to M$, associated with the tangent bundle $TM \to M$ and a Riemmanian metric $g$ on $M$.
\begin{definition}\label{def3.1} We say that a Riemmanian metric $g$ on $M$ is:
\begin{itemize}
\item{\sf of the gradient type} or {\sf non-trapping}, if the geodesic vector field $v^g$ on $SM$ is traversing,
\item {\sf boundary generic}, if the geodesic vector field $v^g$ on $SM$ is boundary generic (in the sense of Definition \ref{def2.2}) with respect to the boundary $\d(SM)$,
\item {\sf traversally generic}, if the geodesic vector field $v^g$ on $SM$ is traversally generic with respect to $\d(SM)$.\footnote{In particular, traversally generic metrics are boundary generic and of the gradient type.} 
\hfill $\diamondsuit$
\end{itemize}
\end{definition}

Since $v^g$ depends smoothly on $g$ and since the traversally generic fields form an open set in the space of all smooth vector fields on $SM$ (\cite{K5}, Theorem 2.2), we conclude that the traversally generic metrics $\mathcal M^\ddagger(M)$ form an open set in the space $\mathcal M(M)$ of all smooth Riemannian metrics on $M$. The question whether, for a given $M$, the space $\mathcal M^\ddagger(M)$ is nonempty remains wide open. For example, if $\d M$ is strictly convex, and $g$ is non-trapping, then $g$ is traversally generic. 

We conjecture that $\mathcal M^\ddagger(M)$ is actually \emph{dense} in the space of non-trapping metrics. 
\smallskip

Now, for a smooth compact manifold $M$ with boundary, we form its double $DM =_{\mathsf{def}} M \cup_{\d M}M$. We denote by $\rho_M: DM \to DM$ a smooth involution such that the quotient space $DM/\rho_M$ is homeomorphic to $M$. Let $$D(SM) =_{\mathsf{def}} SM \cup_{SM|_{\d M}} SM.$$ Let  $\rho_{SM}: D(SM) \to D(SM)$ be an involution such that the quotient $D(SM)/\rho_{SM}$ is homeomorphic to $SM$ and the fixed set $D(SM)^{\rho_{SM}} = SM|_{\d M}$, the restriction of the bundle $SM \to M$ to $\d M \subset M$. The construction of $D(SM)$ gives rise to a spherical fibration $q: D(SM) \to DM$ over $DM$.

{\it Warning:} $\rho_{SM}$ is not induced by the differential of $\rho_M$, and $q$ is not the tangent spherical bundle of $DM$, i.e., $D(SM) \neq S(DM)$! \smallskip

For a compact connected smooth Riemannian $n$-manifold $M$ with boundary, any \emph{traversally generic} metric $g$, via its geodesic flow $v^g$,  defines a stratification of the space $SM$ by the pure strata $\big\{SM(v^g, \om)\big\}_{\om \in \mathbf\Omega^\bullet_{' \langle 2n-2]}}$, which organize the $v^g$-trajectories by their tangency to $\d(SM)$ patterns $\{\om\}$. Let $$SM^\circ(v^g, \om) =_{\mathsf{def}} SM(v^g, \om) \cap \textup{int}(SM).$$
In turn, these strata generate the filtration
\begin{eqnarray}
\Big\{SM^\circ_{-(k+1)}(g) =_{\mathsf{def}} 
 \bigcup_{\om \in \mathbf\Omega^\bullet \big|\, |\om|' \geq k+1} SM^\circ(v^g, \om)\Big\}_k
 \end{eqnarray}
 of $SM \setminus \d(SM)$ and the filtration
\begin{eqnarray}
\Big\{SM_{-(k+1)}(g) =_{\mathsf{def}} 
 \Big(\bigcup_{\om \in \mathbf\Omega^\bullet \big|\, |\om|' \geq k+1} SM^\circ(v^g, \om)\Big) \nonumber \\
  \bigcup \Big( \bigcup_{\om \in \mathbf\Omega^\bullet \big|\, |\om|' \geq k} \d(SM) \cap SM(v^g, \om)\Big)\Big\}_k
 \end{eqnarray}
 of $SM$. The filtration $\{SM_{-(k+1)}(g)\}$ induces the $\rho_{SM}$-equivariant filtration $\{D(SM)_{-(k+1)}(g)\}$ of the double $D(SM) \approx S(DM)$.\smallskip

 \smallskip

For any commutative ring $\mathsf R$, we introduce the free  $\mathsf R$-modules: 
\[
\mathsf C_{\mathbf\mho}^{2n-1-j}\big(D(SM), g; \mathsf R\big) =_{\mathsf{def}}\; H^{2n-1-j}\big(D(SM)_{-j}(g),\, D(SM)_{-(j+1)}(g);\; \mathsf R \big)
\]
and 
\[
\mathsf C_{\mathbf\mho}^{2n-1-j}\big(SM^\circ, g; \mathsf R\big) =_{\mathsf{def}}\;
 H^{2n-1-j}\big(SM_{-j}(g),\; SM_{-(j+1)}(g) \cup \big(SM_{-j}(g) \cap \d(SM)\big);\; \mathsf R \big).
\]

So each traversally generic metric $g$ on $M$ gives rise to a {\sf differential complex} 
\begin{eqnarray}\label{eq11.A}
\mathbf C_{\mathbf\mho}^\ast\big(D(SM),\, g;\, \mathsf R\big) =_{\mathsf{def}} 
\Big\{0 \to \mathsf C_{\mathbf\mho}^{0}\big(D(SM), g; \mathsf R\big) \stackrel{\delta_0}{\longrightarrow} \mathsf C_{\mathbf\mho}^{1}\big(D(SM),\, g;\; \mathsf R\big) \stackrel{\delta_1}{\longrightarrow} \dots \nonumber \\ 
\dots \stackrel{\delta_{2n-2}}{\longrightarrow} \mathsf C_{\mathbf\mho}^{2n-1}\big(D(SM), g; \mathsf R\big) \to 0\Big\}, 
\end{eqnarray}
where the differentials $\{\delta_j\}$ are the boundary homomorphisms from the long exact cohomology sequences of the triples 
$$\big\{D(SM)_{-(j -1)}(g) \supset  D(SM)_{-j}(g) \supset D(SM)_{-(j +1)}(g)\big\}_j.$$

Similarly, $g$  produces the differential complex 
\begin{eqnarray}\label{eq11.B}
\mathbf C_{\mathbf\mho}^\ast\big(SM^\circ, g;\, \mathsf R\big) =_{\mathsf{def}} 
\Big\{0 \to \mathsf C_{\mathbf\mho}^{0}\big(SM^\circ, g; \mathsf R\big) \stackrel{\delta_{0}}{\longrightarrow} \mathsf C_{\mathbf\mho}^{1}\big(SM^\circ, g; \mathsf R\big) \stackrel{\delta_{1}}{\longrightarrow} \dots \nonumber \\ 
\dots \stackrel{\delta_{2n-2}}{\longrightarrow} \mathsf C_{\mathbf\mho}^{2n-1}\big(SM^\circ, g; \mathsf R\big) \to 0\Big\},
\end{eqnarray}
where the differentials $\{\delta_j\}$ are the boundary homomorphisms from the long exact homology sequences of the triples $$\big\{SM_{-(j -1)}(g) \cup \d(SM)\, \supset \, SM_{-j}(g)  \cup \d(SM)\, \supset \, SM_{-(j +1)}(g) \cup \d(SM)\big\}_j.$$

We denote by $$\mathsf B_{\mathbf\mho}^{2n-1-j}\big(SM^\circ, g; \mathsf R\big) \subset \mathsf C_{\mathbf\mho}^{2n-1-j}\big(SM^\circ, g; \mathsf R\big)$$ the image of the differential $\delta_{2n-2-j}$ from (\ref{eq11.B}). Similarly, let $$\mathsf B_{\mathbf\mho}^{2n-1-j}\big(D(SM), g; \mathsf R\big) \subset \mathsf C_{\mathbf\mho}^{2n-1-j}\big(D(SM), g; \mathsf R\big)$$ stand for the image of the differential $\delta_{2n-2-j}$ from (\ref{eq11.A}).

\begin{conjecture}\label{conj11.A} Let $g$ be a traversally generic Riemmanian metric on a compact connected smooth $n$-manifold $M$ with boundary. Then, as described above, the metric $g$ generates the differential complexes $\mathbf C_{\mathbf\mho}^\ast\big(SM^\circ, g; \mathsf R\big)$ and $\mathbf C_{\mathbf\mho}^\ast\big(D(SM), g; \mathsf R\big)$ of free $\mathsf R$-modules. 
\smallskip

We conjecture that the homology groups of these differential complexes depend only on the connected component of the space of traversally generic metrics on $M$, to which $g$  belongs. 
\hfill $\diamondsuit$
\end{conjecture}

The next theorem claims that, for a non-trapping metric $g$, the differential complexes $\mathbf C_{\mathbf\mho}^\ast\big(SM^\circ, g\big)$ and $\mathbf C_{\mathbf\mho}^\ast\big(D(SM), g\big)$ can be reconstructed from the scattering map $C_{v^g}$. We call such reconstructions ``{\sf holographic}" since the objects that are affiliated with the $(2n-1)$-dimensional bulk $SM$ or $S(DM)$ and the geodesic flow are recorded on the pair $\d_1^+(SM),\hfill \break \d_1^-(SM)$ of diffeomorphic $(2n-2)$-dimensional screens.

\begin{theorem}\label{th11.15} Let $g$ be a traversally generic Riemmanian metric on a compact connected smooth $n$-manifold $M$ with boundary. \smallskip

The differential complexes $\mathbf C_{\mathbf\mho}^\ast\big(SM^\circ, g; \mathsf R\big)$ and $\mathbf C_{\mathbf\mho}^\ast\big(D(SM), g; \mathsf R\big)$ can be reconstructed from the geodesic scattering map $C_{v^g}: \d_1^+(SM) \to \d_1^-(SM)$.
\end{theorem}

\begin{proof} Let $\mathcal F(v^g)$ be the oriented 1-dimensional foliation on $SM$, produced by the $v^g$-flow. Any traversally generic geodesic field $v^g$ is automatically boundary generic and of the gradient (non-trapping) type. If $g$ is non-trapping and the geodesic field $v^g$ is boundary generic, then by Theorem 3.3 from \cite{K5}, the scattering map $C_{v^g}: \d_1^+(SM) \to \d_1^-(SM)$ allows for a reconstruction of the pair $(SM, \mathcal F(v^g))$, up to a homeomorphism of $SM$ which is the identity on the boundary $\d(SM)$. That  homeomorphism preserves the stratification $\mathcal S^\bullet_{v^g}(SM)$, whose strata are the connected components of the stratification $\mathcal S_{v^g}(SM) = \{\d(SM)(v^g, \om),\; SM^\circ(v^g, \om)\}_\om$. Therefore, the topological types of the stratifications $\mathcal S^\bullet_{v^g}(SM^\circ)$ and $\mathcal S^\bullet_{v^g}(D(SM))$  are determined by $C_{v^g}$. As a result, the differential complexes $\mathbf C_{\mathbf\mho}^\ast\big(SM^\circ, g; \mathsf R\big)$ and $\mathbf C_{\mathbf\mho}^\ast\big(D(SM), g; \mathsf R\big)$, whose construction depends only on the \emph{stratified} topological types of the spaces $SM$ and $D(SM)$, can be reconstructed from the geodesic scattering map $C_{v^g}$ along the lines of \cite{K4}.
\hfill 
\end{proof}

Abusing notations, we denote by $C_{\mathbf\mho}^j\big(D(SM), g; \Z \big)$ the image of $C_{\mathbf\mho}^j\big(D(SM), g; \Z \big)$ in the vector space $C_{\mathbf\mho}^j\big(D(SM), g; \R\big)$, viewed as an integral lattice. Similarly, we consider the integral lattice $C_{\mathbf\mho}^j\big(SM^\circ, g; \Z \big) \subset C_{\mathbf\mho}^j\big(SM^\circ, g; \R \big)$. The integral lattice $C_{\mathbf\mho}^j\big(SM, g; \Z \big)$ comes equipped with a \emph{basis} whose vectors correspond to the connected components of the strata 
$$\big\{SM(v^g, \om) \cap \textup{int}(SM)\big\}_{\{\om \in \mathbf\Omega^\bullet \big | 
|\om|' = j\}}.$$   

Similarly, the integral lattice $C_{\mathbf\mho}^j\big(D(SM), g; \Z \big)$ comes equipped with a basis whose vectors correspond to the connected components of the strata 
$$\big\{SM(v^g, \om)\big\}_{\{\om \in \mathbf\Omega^\bullet \big | |\om|' = j\}},$$ each stratum being considered twice, together with the connected components of the strata $$\big\{SM(v^g, \om) \cap \d(SM)\big\}_{\{\om \in \mathbf\Omega^\bullet \big | 
|\om|' = j - 1\}}.$$  

Using these bases, we introduce the $l_1$-norms $\| \sim \|_{\mathbf\mho}$ in the vector spaces  $C_{\mathbf\mho}^j\big(SM^\circ, g; \R \big)$ and $C_{\mathbf\mho}^j\big(D(SM), g; \R \big)$ so that the basic vectors (which belong to the lattices $C_{\mathbf\mho}^j\big(SM^\circ, g; \Z \big)$ and $C_{\mathbf\mho}^j\big(D(SM), g; \Z \big)$, respectively) have lengths $1$. The unit balls in these norms are the perfect polyhedra, dual to the hypercubes in the corresponding spaces. \smallskip

The semi-norms $\| \sim \|_{\mathbf\mho}$ induce ``honest" norms $\| \sim \|_{\mathbf\mho}^\bullet$ in the quotient spaces\footnote{Recall that the ``quotient norm" of a given vector $\vec V$ in the quotient space is defined to be the infimum of the norms of all the vectors (in the original space) that represent $\vec V$.} $$C_{\mathbf\mho}^j\big(SM^\circ, g; \R \big)\big/ B_{\mathbf\mho}^j\big(SM^\circ, g; \R \big) \;\; \text{and} \; \; C_{\mathbf\mho}^j\big(D(SM), g; \R \big)\big/ B_{\mathbf\mho}^j\big(D(SM), g; \R \big),$$ respectively.\smallskip 

We denote by $\diamondsuit_{\mathbf\mho}^j\big(SM^\circ, g \big)$ and 
$\diamondsuit_{\mathbf\mho}^j\big(D(SM), g \big)$ the unit balls in these quotient norms $\| \, \sim \, \|_{\mathbf\mho}^\bullet$. They are convex hulls of the images, under the quotient maps, of the verticies of the perfect polyhedron $\{\| \, \sim \, \|_{\mathbf\mho} = 1\}$.\smallskip

The manifolds $SM$ and $S(DM)$ are orientable. 
So the Poincar\'{e} Duality is available for their homology and cohomology with  coefficients in $\R$ or $\Z$.  As in \cite{K3} (where we dealt with arbitrary traversally generic flows), for each $j$, we consider the \emph{localized} Poincar\'{e} Duality operators over the coefficient rings $\mathsf R = \Z, \R$ \footnote{We have suppressed in (\ref{eq11.AA}) the dependence of homology and cohomology on the coefficients $\mathsf R$.}:
%
\[
\mathcal L_{j}(g): H_{j}(D(SM)) \stackrel{\approx \mathcal D}{\longrightarrow}  H^{2n -1 -j}(D(SM)) \stackrel{i^\ast_{\mathsf{loc}}}{\longrightarrow}  H^{2n-1 -j}\big(D(SM)_{-j}(g)\big)
\]
\[
\; \approx \; 
C_{\mathbf\mho}^{2n -1 -j}(D(SM), g)\Big/ B_{\mathbf\mho}^{2n -1 -j}(D(SM), g),
\]
\[
\mathcal M_{j}(g): H_{j}(SM)\stackrel{\approx \mathcal D}{\longrightarrow}  H^{2n -1 -j}(SM, \d(SM))  
\stackrel{i^\ast_{\mathsf{loc}}}{\longrightarrow}  \\
 H^{2n-1 -j}\big(SM_{-j}(g), SM_{-j}(g) \cap \d(SM)\big)
 \]
 \[
\;  \approx \;  C_{\mathbf\mho}^{2n -1 - j}(SM^\circ, g)\Big/ B_{\mathbf\mho}^{2n -1 - j}(SM^\circ, g).  
 \]
\begin{eqnarray}\label{eq11.AA}
\end{eqnarray}
Here the natural homomorphisms $i^\ast_{\mathsf{loc}}$ are induced by the inclusions of the strata $$SM^\circ_{-j}(g) \subset  SM \; \text{and} \; D(SM)_{-j}(g) \subset D(SM),$$ thus the term ``localized" in the names of the two operators. 
\smallskip

The RHS of isomorphisms ``$\approx$" in (\ref{eq11.AA}) can be justified exactly by the same homological argument as in \cite{K3}, page 516;  just replace an arbitrary traversally generic vector fields $v$ on $X$ with the geodesic traversally generic vector fields $v^g$ on $SM$. \smallskip   

$\mathcal L_{j}(g; \R)$ maps the lattice $H_{j}(D(SM); \Z)$ to the lattice  $H^{2n-1 -j}\big(D(SM)_{-j}(g); \Z\big)$.  Similarly, $\mathcal M_{j}(g; \R)$ maps the lattice $H_{j}(SM; \Z)$ to the lattice $H^{2n-1 -j}\big(SM^\circ_{-j}(g); \Z\big)$.\smallskip

We denote by $\tilde{\mathsf B}^{\mathbf \D}_{j}(D(SM)) \subset H_{j}(D(SM); \R)$ and $\tilde{\mathsf B}^{\mathbf \D}_{j}(SM)  \subset H_{j}(SM; \R)$ the set of vectors whose simplicial semi-norms $|\sim|_{\mathbf \D}$ do not exceed $1$. These are the ``unit balls". Let $\d\tilde{\mathsf B}^{\mathbf \D}_{j}(D(SM))$ and $\d\tilde{\mathsf B}^{\mathbf \D}_{j}(SM)$ denote the sets of vectors whose semi-norms are equal $1$  (these ``spheres" may not be compact!).\smallskip

Note that, for any pair of vectors $V$ and $W$ such that $|W|_{\mathbf\D} = 0$, we get $|V+W|_{\mathbf\D} = |V|_{\mathbf\D}$.  Therefore, $| \sim |_{\mathbf\D}$ becomes a {\it norm} on the quotient $H^{\mathbf \D}_\ast(\sim; \R)$ of the homology space $H_\ast(\sim; \R)$ by the subspace $H^{\{\| \sim\|_{\mathbf \D} = 0\}}_\ast(\sim; \R)$ of vectors whose simplicial semi-norms vanish. \smallskip 

So we may form the compact convex ball $\mathsf B^{\mathbf \D}_{j}(D(SM)) \subset H^{\mathbf \D}_{j}(D(SM); \R)$, the image of $\tilde{\mathsf B}^{\mathbf \D}_{j}(D(SM))$ under the quotient map $H_{j}(D(SM); \R) \to H^{\mathbf \D}_{j}(D(SM); \R)$; similarly, we may form the compact convex ball $\mathsf B^{\mathbf \D}_{j}(SM) \subset H^{\mathbf \D}_{j}(SM; \R)$, the image of $\tilde{\mathsf B}^{\mathbf \D}_{j}(D(SM))$ under the quotient map $H_{j}(SM; \R) \to H^{\mathbf \D}_{j}(SM; \R)$.\smallskip

We will use the localized Poincar\'{e} Duality operators $\mathcal L_{j}(g; \R)$ and  $\mathcal M_{j}(g; \R)$ from (\ref{eq11.AA}) to project linearly the unit balls $\tilde{\mathsf B}^{\mathbf \D}_{j}(D(SM))$ and $\tilde{\mathsf B}^{\mathbf \D}_{j}(SM)$ on the $g$-dependent ``screens" $$C_{\mathbf\mho}^{2n -1 -j}(D(SM), g; \R)\big/ B_{\mathbf\mho}^{2n -1 -j}(D(SM), g; \R)$$ and $$C_{\mathbf\mho}^{2n -1 - j}(SM^\circ, g; \R)\big/ B_{\mathbf\mho}^{2n -1 - j}(SM^\circ, g; \R),$$ respectively. These screens are manufactured with the help various metrics $g \in \mathcal G^\ddagger(M)$. \smallskip

The next theorem, one of our main results, makes several claims about the geometry of these projections. 

\begin{theorem}\label{th11.16} Let $M$ be a compact connected smooth $n$-manifold with boundary, where $n \geq 3$. Let $j \in [0, n]$.

\begin{itemize}
\item Assume that, for each connected component of the boundary $\d M$, the image of its fundamental group in $\pi_1(DM)$ is amenable.  
\smallskip

Then there is a universal constant $\lambda = \lambda(n, j) \geq 1$ such that, for every $M$ and every traversally generic Riemannian metric $g$ on $M$, the image of the unit (in the simplicial semi-norm) sphere $\d \tilde{\mathsf B}^{\mathbf\D}_{j}(D(SM))$ 
under the localized Poincar\'{e} Duality operator $\mathcal L_{j}(g; \R)$, is contained in the complement to the radius $\lambda^{-1}$ ball $$\lambda^{-1} \cdot \diamondsuit_{\mathbf\mho}^{2n-1 -j}\big(D(SM), g \big).$$  

\item Assume that, for each connected component of the boundary $\d M$, the image of its fundamental group in $\pi_1(M)$ is amenable. 
\smallskip

Similarly,  there is a universal constant $\mu = \mu(n, j) \geq 1$ such that, for every $M$ and every traversally generic Riemannian metric $g$ on $M$, the image of the unit sphere $\d \tilde{\mathsf B}^{\mathbf\D}_{j}(SM)$ 
under the localized Poincar\'{e} Duality operator $\mathcal M_{j}(g; \R)$ is contained in the complement to the radius $\mu^{-1}$ ball $$\mu^{-1} \cdot \diamondsuit_{\mathbf\mho}^{2n-1-j}\big(SM, g \big).$$
\end{itemize}
\end{theorem} 

\begin{proof} We need to show that the hypotheses of Theorem \ref{th11.16} imply the validity of the hypotheses of Theorem 4.2  from \cite{K3} (see also Theorem 8.8 from \cite{K7}).  
The latter theorem applies to any traversally generic vector field $v$ on a connected compact smooth manifold $X$ with boundary; in particular, it applies to any traversally generic geodesic vector field $v^g$ on $SM$.\smallskip

Let us first discuss Theorem 4.2  from \cite{K3}, the foundation of this proof, at some length. In fact, that theorem is stated in \cite{K3} or \cite{K7} in a slightly different and less geometrical form than Theorem \ref{th11.16} here: it claims the validity of the inequality  $\| h \|_{\mathbf \D} \leq \lambda \cdot \| \mathcal L_{j}(g)(h) \|_{\mathbf\mho}^\bullet$ for every class $h \in H_j(D(SM); \R)$, and of the inequality  $\| h \|_{\mathbf \D} \leq \mu \cdot \| \mathcal M_{j}(g)(h)\|_{\mathbf\mho}^\bullet$ for every class $h \in H_j(SM; \R)$, where $\lambda \geq \mu \geq 1$ being some universal constants. These constants depend only on $\dim(M)$ and the index $j \in [0, 2n-1]$. We will see soon that, in fact, one gets non-trivial results only for $j \in [0, n]$.\smallskip

The inequality $\| h \|_{\mathbf \D} \leq \lambda \cdot \| \mathcal L_{j}(g)(h) \|_{\mathbf\mho}^\bullet$ may be interpreted as claiming  that the $\mathcal L_{j}(g; \R)$-image of the unit sphere $\d \tilde{\mathsf B}^{\mathbf\D}_{j}(D(SM))$ (in the simplicial semi-norm) is contained in the complement to the ball $\lambda^{-1} \cdot \diamondsuit_{\mathbf\mho}^{2n-1 -j}\big(D(SM), g \big)$ of radius $\lambda^{-1}$ in the target space. A similar interpretation is available for $\mathcal M_{j}(g; \R)$.\smallskip

Let us describe briefly the source of the universal constants $\lambda, \mu$, participating in these inequalities\footnote{This is the only place where the hypotheses that $v^g$ is traversally generic  (and not only boundary generic and of the gradient type) seems to be important!}. Recall that, in Section 2, we have introduced the model space $\mathsf E_{\om} \subset \R \times \R^{|\om|'} \times \R^{n-|\om|'}$, given by the polynomial inequality $\{\wp (u, \vec x) \leq 0\}$, where $\wp$ is defined by the LHS of (\ref{eq2.4}) and $\vec x \in \R^{|\om|'}$. 

Consider the obvious projection $p: \R \times \R^{|\om|'}\times \R^{n-|\om|'} \to \R^{|\om|'}\times \R^{n-|\om|'}$. The fibers of the projection $p: \mathsf E_{\om} \to \R^{|\om|'}\times \R^{n-|\om|'}$ are unions the $\d_u$-trajectories  in $\mathsf E_{\om}$. Therefore $\mathsf E_{\om}$ acquires a stratification 
$\big\{ \mathsf E_{\om} (\d_u, \hat\om) \big\}_{\hat\om}$, labeled by the elements $\hat\om \succeq \om$. 
They form  the sub-poset $\om_{\preceq} \subset \mathbf\Omega^\bullet_{'\langle 2n-2]}$. 
Put $$\d\mathsf E_{\om}(\d_u, \hat\om) =_{\mathsf{def}} \mathsf E_{\om}(\d_u, \hat\om) \cap \d\mathsf E_{\om}$$ and $$\mathsf E^\circ_{\om}(\d_u, \hat\om) =_{\mathsf{def}} \mathsf E_{\om}(\d_u, \hat{\om}) \cap \textup{int}(\mathsf E_{\om}).$$ We denote $\mathcal S^\bullet(\mathsf E_{\om}^\circ)$ the stratification of the space $\mathsf E^\circ_{\om}$  by the connected components of these strata.

The double $D\mathsf E_{\om}$ of $\mathsf E_{\om}$ is stratified by the connected components of $\big\{\d\mathsf E_{\om}(\d_u, \hat\om)\big\}_{\hat\om \in \om_{\prec}}$, together with the connected components of $\big\{\mathsf E^\circ_{\om}(\d_u, \hat\om)\big\}_{\hat\om \in \om_{\preceq}}$ and their images under the involution $\tau: D(\mathsf E_{\om}) \to D(\mathsf E_{\om})$ that is a part of the doubling construction.  We denote by $\mathcal S^\bullet(D(\mathsf E_\om))$ this stratification. 
\smallskip 

The universal constant $\mu$ is the maximum of the $\mathcal S^\bullet(\mathsf E_{\om}^\circ)$-{\sf stratified}\footnote{See \cite{Gr1} and \cite{AK} for an accurate definition of the stratified simplicial norm.} relative (to their boundaries) simplicial norms of small $j$-disks, each one being normal to a particular connected component of the $(2n+1-j)$-dimensional strata $\{\mathsf E^\circ_{\om}(\d_u, \hat\om)\}_{\om, \,\hat\om \in \om_{\preceq}}$. The maximum being taken over all pairs $\om \preceq \hat \om$, where $\om \in \mathbf\Omega^\bullet_{'\langle 2n-2]}$.

The universal constant $\lambda$ is the maximum of the $\mathcal S^\bullet(D(\mathsf E_\om))$-{\sf stratified} relative (to their boundaries) simplicial norms of small $j$-disks, each one being normal to a particular connected component of the $(2n+1-j)$-dimensional strata $$\Big\{\mathsf E^\circ_{\om}(\d_u, \hat\om),\; \tau(\mathsf E^\circ_{\om}(\d_u, \hat\om)),\; \d \mathsf E_{\om}(\d_u, \hat\om) \Big\}
$$ 
in the double $D(\mathsf E_{\om})$. The maximum is taken over all pairs $\om \preceq \hat \om$, where $\om \in \mathbf\Omega^\bullet_{'\langle 2n-2]}$.
%
\smallskip

Now we are ready to verify that the hypotheses of Theorem \ref{th11.16} imply the hypotheses of Theorem 4.2 from \cite{K3}. \smallskip

Consider the tangent $(n-1)$-spherical fibration $p: SM \to M$ and its double, the fibration $q: D(SM) \to DM$. We denote by $\d_\a M$ a typical connected component of $\d M$. 

Using the long homotopy sequences of the fibration $p$ and of its restriction to $\d_\a M$, we notice that $p_\ast: \pi_1(p^{-1}(\d_\a M)) \approx \pi_1(\d_\a M)$ and $p_\ast: \pi_1(SM) \approx \pi_1(M)$ since $\pi_1(S^{n-1}) = 0$ for $n\geq 3$ and $\pi_0(\text{spherical fiber}) \approx \pi_0(\text{total space})$.  

Consider the square diagram that is formed by these four fundamental groups, where the two vertical homomorphisms are induced by the inclusions $\d_\a M \hookrightarrow M$ and $p^{-1}(\d_\a M) \hookrightarrow SM$. Using its commutativity, we conclude that, if the image of $\pi_1(\d_\a M) \to \pi_1(M)$ is an amenable group, so is the image of  $\pi_1(p^{-1}(\d_\a M)) \to \pi_1(SM)$. Similarly, if the image of $\pi_1(\d_\a M) \to \pi_1(DM)$ is an amenable group, so is the image of  $\pi_1(q^{-1}(\d_\a M)) \to \pi_1(D(SM))$. \smallskip

By the Universal Coefficient Theorem, the natural homomorphism $\mu: H_k(X; \Z) \to H_k(X; \R)$ generates an isomorphism $H_k(X; \Z) \otimes \R \approx H_k(X; \R)$ for any $CW$-complex $X$. So it suffices to check the desired properties of the simplicial norm on the integral lattice $\mu(H_k(X; \Z)) \subset H_k(X; \R)$, where $X = M, SM, DM, D(SM)$, and extend them by linearity. 
\smallskip 

Recall that an odd multiple $[f]$ of every integral homology class $[h] \in H_k(M; \Z)$ can be realized by a continuous map $f: N \to M$, where $N$ is a closed orientable $k$-dimensional smooth (and thus $\mathsf{PL}$-) manifold \cite{CF}. 
Each singular cycle $f: N \to M$ induces a spherical fibration $f^\ast(p): SN \to N$, the pull-back of the fibration $p: SM \to M$ under $f$. Thus $f$ induces the canonical map $F_f: SN \to SM$, which is viewed as a singular cycle on $SM$. We denote by $[F_f] \in H_{k+ n-1}(SM; \Z)$ the homology class of the singular cycle $F_f$.  

We claim that, for $n \geq 3$ and any $f$, the simplicial semi-norm $\| [F_f]\|_{\mathbf \D} =0$.
Indeed, 
 by an argument above, $p_\ast: \pi_1(SM) \to \pi_1(M)$ is an isomorphism.

We denote by $P: BM \to M$ the disk bundle, associated with the bundle $p: SM \to M$. Let $f^\ast(P): BN \to N$ be the pull-back of $P$ under $f$. So we get a canonical map $G_f: BN \to BM$. 

The space $SM$ of the spherical bundle $p$ is the boundary of the disk bundle $P: BM \to M$, and the space $SN$ of the spherical bundle $f^\ast(p)$ is the boundary of the disk bundle $f^\ast(P): BN \to N$. 
Therefore the singular cycle 
$H_f: SN  \stackrel{F_f}{\rightarrow} SM \stackrel{p}{\rightarrow} M$ is the boundary of a singular chain $BN  \stackrel{G_f}{\rightarrow} BM \stackrel{P}{\rightarrow} M$, and thus $[H_f] = p^\ast([F_f])$ vanishes in $H_{k + n - 1}(M; \Z)$. By \cite{Gr}, continuous maps of spaces that induce isomorphisms of the fundamental groups are \emph{isometries} in the simplicial semi-norms $\| \sim \|_{\mathbf\D}$ in their homologies. Therefore, using that $p$ induces isometries, we conclude that $\| \, [H_f] \,\|_{\mathbf\D} = \| \, \mathbf 0 \,\|_{\mathbf\D}= 0$ implies $\| \, [F_f] \,\|_{\mathbf\D} = 0$. 
\smallskip

Since $M$ is of a homotopy type of an $(n-1)$-dimensional $CW$-complex (recall that $\d M \neq \emptyset$), the fibration $p: SM \to M$ admits a section $\s: M \to SM$. With the help of $\s$, $H_\ast(M; \Z)$ is direct summand of $H_\ast(SM; \Z)$. By the Leray-Hirsh Theorem (see Theorem 4D.1 in \cite{Hat}), the existence of $\s$ implies the isomorphism $H_\ast(SM; \Z) \approx H_\ast(S^{n-1};\Z) \otimes H_\ast(M;\Z)$. In other words, $H_k(SM; \Z) \approx H_k(M; \Z) \oplus H_{k - n +1}(M; \Z)$ for all $k \in [0, 2n-1]$. The first summand is delivered by $\s_\ast: H_\ast(M;\Z) \to H_\ast(SM; \Z) $. Up to an odd multiple, the homology classes of the second summand, viewed as elements of $H_k(SM; \Z)$ that belong to $\ker(p_\ast)$, are realizable by singular cycles of the form $F_f$.  By the previous argument, their simplicial norms vanish. 
\smallskip

Therefore we get the canonical \emph{isomorphism} $p_\ast: H^{\mathbf\D}_\ast(SM) \approx H^{\mathbf\D}_\ast(M)$ of the two quotient spaces, viewed as \emph{normed spaces}. In particular, $H^{\mathbf\D}_\ast(SM) = 0$ for $j \geq n$.
\smallskip

Similar arguments, applied to the spherical fibration $q: D(SM) \to DM$, lead to the isometry $q_\ast: H^{\mathbf\D}_k(D(SM)) \approx H^{\mathbf\D}_k(DM)$ for all $k \in [0, n]$. Indeed, the spherical fibration $q: D(SM) \to DM$ admits a section $D\s$, the double of the section $\s$ of $p$. (Note that the tangent spherical bundle $S(DM) \to DM$ may not have a section.) In particular, $H^{\mathbf\D}_\ast(D(SM)) = 0$ for $j > n$.




Thus we proved that the hypotheses of Theorem \ref{th11.16} imply the validity of the hypotheses of Theorem 4.2  from \cite{K3}.
This completes the proof of the theorem.
\hfill 
\end{proof}
\begin{corollary}\label{corAAA} Let $(N, g)$ be a closed connected smooth Riemannian  
$n$-manifold, where $n \geq 3$. Let $U \subset N$ be a codimension zero submanifold with a smooth boundary so that $U$ is contained in a topological $n$-ball. Put $M =_{\mathsf{def}} N \setminus \textup{int}(U)$, and let us assume that the metric $g|_M$ is traversally generic. 

Then the number of connected codimension $n$ components of the $\mathcal S^\bullet_{v^g}(D(SM))$-stratification of $D(SM)$ exceeds $\lambda^{-1} \cdot \| [DM] \|_{\mathbf\D}$, where $\lambda \geq 1$ depends only on $n$. 

In particular, if $\| [DM] \|_{\mathbf\D} \neq 0$, there exists a $(n-1)$-dimensional family of geodesics $\g$ in $M$ such that their reduced multiplicity $m'(\g) = n$ (see formula (\ref{eq2.1})). \hfill $\diamondsuit$
\end{corollary}

\noindent{\bf Remark 3.1.} Note that, by the previous arguments, $\| [D(SM)] \|_{\mathbf\D} = 0$. Therefore Theorem \ref{th11.16} does not tell anything about the number of codimension $2n -1$ connected components of the $\mathcal S^\bullet_{v^g}(D(SM))$-stratification of $D(SM)$. These arise from the finitely many geodesics $\g$ that have the \emph{maximal} reduced multiplicity $m'(\g) = 2n-2$ to the boundary $\d M$. 
\hfill $\diamondsuit$ \smallskip

\noindent {\bf Example 3.1.} Let $(N, g)$ be a closed hyperbolic $n$-manifold, $n \geq 3$. 
Let  $U$ be a smooth $n$-ball in $N$. Put $M = N \setminus U$. We assume that $g|_M$ is traversally generic. 

Then $DM = N \# N$, where ``$\#$" stands for the connected sum. For $n \geq 3$, we get $\| [DM] \|_{\mathbf\D} = 2\cdot \| [N] \|_{\mathbf\D}$ (\cite{Gr}, Section 3.5). By the hyperbolicity of $N$,  $$\| [N] \|_{\mathbf\D} = vol_{\mathsf{hyp}}(N)/vol_{\mathsf{hyp}}(\D^n),$$ the normalized hyperbolic volume of $N$. Here $vol_{\mathsf{hyp}}(\D^n)$ stands for the hyperbolic volume of an ideal simplex $\D^n$ in the hyperbolic space. In such a case, the number of codimension $n$ connected components of the $\mathcal S^\bullet_{v^g}(D(SM))$-stratification of $D(SM)$ exceeds $$2\big(\lambda \cdot vol_{\mathsf{hyp}}(\D^n)\big)^{-1} \cdot vol_{\mathsf{hyp}}(N) > 0.$$ 

In particular,  there exists a $(n-1)$-dimensional family of geodesics $\g$ in $M$ such that their reduced multiplicity $m'(\g) = n$.

Therefore, for a given family of closed hyperbolic $n$-manifolds $\{N_k\}_{k \to \infty}$ of increasing volumes and the corresponding compact hyperbolic manifolds $M_k = N_k \setminus U_k$, the number of codimension $n$ connected components of the $\mathcal S^\bullet_{v^g}(D(SM))$-stratification of $D(SM_k)$ grows at least as fast as $\{vol_{\mathsf{hyp}}(N_k)\}_{k \to \infty}$.\smallskip
\hfill $\diamondsuit$ 
\smallskip

\begin{theorem}\label{th11.16a} For any metric $g \in \mathcal G^\ddagger(M)$, the localized Poincar\'{e} Duality operators $\mathcal L_{j}(g; \R)$ and $\mathcal M_{j}(g; \R)$ in (\ref{eq11.AA}), as well as the simplicial semi-norms $\| \sim \|_{\mathbf \D}$ in their source spaces, and the norms $\|\, \sim \, \|_{\mathbf\mho}^\bullet$ in their target spaces, can be reconstructed from the geodesic scattering map $C_{v^g}: \d_1^+(SM) \to \d_1^-(SM)$.  
\end{theorem}

\begin{proof}  By Theorem 3.3 from \cite{K5} (see also \cite{K6}), the scattering map $C_{v^g}$ allows for a reconstruction of the pair $(SM, \mathcal F(v^g))$, up to a stratification-preserving homeomorphism of $SM$ which is the identity on $\d(SM)$. Therefore, with the help of the involution $\tau: D(SM) \to D(SM)$, we get also a reconstruction of the stratified topological type of the double $D(SM)$. 

The localized Poincar\'{e} Duality operators $\mathcal L_{j}(g; \R)$ and $\mathcal M_{j}(g; \R)$ also depend only on the $\mathcal S^\bullet_{v^g}(D(SM))$-stratified and
 $\mathcal S^\bullet_{v^g}(SM^\circ)$-stratified topological types of $D(SM)$ and $SM^\circ$, respectively. As a result, $\mathcal L_{j}(g; \R)$ and $\mathcal M_{j}(g; \R)$ can be recovered from $C_{v^g}$.\smallskip

The semi-norm $\| \sim \|_{\mathbf \D}$ of a given homology class is an invariant of the topological type of the underlying space. Therefore (in accordance with Theorem 3.5 from \cite{K5}), the semi-norms $\| \sim \|_{\mathbf \D}$ on the source spaces of $\mathcal L_{j}(g; \R)$ and $\mathcal M_{j}(g; \R)$ can be recovered from $C_{v^g}$. The norm $\|\, \sim \, \|_{\mathbf\mho}^\bullet$ on the target spaces is defined also solely in terms of the stratified topological types of $D(SM)$ and $SM^\circ$, and therefore, by Theorem 3.3 from \cite{K5}, depends on the scattering map $C_{v^g}$ only.
\end{proof}
\begin{theorem}\label{th11.17}  Let $M$ be a compact connected smooth $n$-manifold with boundary, where $n \geq 3$. Let $j \in [0, n]$.
\begin{itemize}
\item Assume that, for each connected component of the boundary $\d M$, the image of its fundamental group in $\pi_1(M)$ is amenable. 

Then, for any traversally generic Riemannian metric $g$ on $M$, the number of $(2n-1 -j)$-dimensional connected components in the stratification $\{SM^\circ(v^g, \om)\}_{\om}$ is greater than or equal to $\textup{rank} \big(H^{\mathbf\D}_{j}(M)\big)$. \smallskip

\item Assume that, for each connected component of the boundary $\d M$, the image of its fundamental group in $\pi_1(DM)$ is amenable. 
\smallskip

Then, for any traversally generic Riemannian metric $g$ on $M$, the number of $(2n-1 -j)$-dimensional connected components in the stratification $\{D(SM)(v^g, \om)\}_{\om}$ is greater than or equal to $\textup{rank} \big(H^{\mathbf\D}_{j}(DM)\big)$. 
\end{itemize}
\end{theorem}

\begin{proof} By Theorem \ref{th11.16}, the kernel $\ker(\mathcal L_j(g))$ is contained in the subspace $$H_j^{\{\mathbf\| \sim \| = 0\}}(D(SM); \R \subset H_j(D(SM); \R).$$ Therefore, employing (\ref{eq11.AA}), we get
\[ 
\textup{rank}(H^{\mathbf \D}_j(D(SM); \R) \leq \textup{rank}\big(H_j(D(SM); \R) \big/ \ker(\mathcal L_j(g)\big) = \textup{rank}(\textup{im}\big(\mathcal L_j(g)\big)
\]
\[ \leq \textup{rank}\big(C_{\mathbf\mho}^{2n -1 -j}(D(SM), g)\Big/ B_{\mathbf\mho}^{2n -1 -j}(D(SM), g)\big) 
\leq \textup{rank}\big(C_{\mathbf\mho}^{2n -1 -j}(D(SM), g)\big),
\]
the number of codimension $j$ connected components in the stratification $\{D(SM)(v^g, \om)\}_{\om}$.
Similarly, 
\[ 
\textup{rank}(H^{\mathbf \D}_j(SM; \R) \leq \textup{rank}\big(C_{\mathbf\mho}^{2n -1 -j}(SM, g)\big),
\]
the number of codimension $j$ connected components in the stratification $\{SM^\circ(v^g, \om)\}_{\om}$. \smallskip

Using that  $p_\ast: H^{\mathbf\D}_j(SM) \approx H^{\mathbf\D}_j(M)$ (for $j \geq n$, both spaces are trivial) and  $q_\ast: H^{\mathbf\D}_j(D(SM)) \approx H^{\mathbf\D}_j(DM)$ (for $j > n$, both spaces are trivial) are isometric, the claim follows.
\hfill 
\end{proof}

The next corollary spells out the claims of Theorem \ref{th11.17} by applying the formula on the bottom of page 537 in \cite{K3} to the settings of the theorem.

\begin{corollary}\label{cor11.A} Let $M$ be a compact connected smooth $n$-manifold with boundary,  and $g$ a traversally generic metric on $M$. Let $n \geq 3$. 

\begin{itemize}
\item Assume that, for each connected component of the boundary $\d M$, the image of its fundamental group in $\pi_1(M)$ is amenable. \smallskip

Then, for each $j \in [0, n-1]$, the space of geodesics $\mathcal T(v^g)$ has the property
\[ \sum_{\om \in \mathbf\Omega \big | |\om|' = j} \sup(\om) \cdot \#\Big(\pi_0\big(\mathcal T(v^g, \om)\big)\Big) \geq \, \textup{rank} \big(H^{\mathbf\D}_{j}(M)\big),
\]
where $\mathcal T(v^g, \om)$ stands for the family of geodesics $\g \subset M$ whose intersections $\g \cap \d M$ generate the combinatorial tangency pattern $\om$,  $\sup(\om) =_{\mathsf{def}} |\om| - |\om|'$, and $\pi_0(\sim)$ denotes the set of path connected components of the appropriate space. 
\smallskip

\item Assume that, for each connected component of the boundary $\d M$, the image of its fundamental group in $\pi_1(DM)$ is amenable. \smallskip

Then, for each $j \in [0, n]$, the space of geodesics $\mathcal T(v^g)$ has the property
\[ \sum_{\om \in \mathbf\Omega \big | |\om|' = j} \sup(\om) \cdot \#\Big(\pi_0\big(\mathcal T(v^g, \om)\big)\Big) \; +
\]
\[+\; 2 \cdot \sum_{\hat\om \in \mathbf\Omega \big | |\hat\om|' = j+1} (\sup(\hat\om) - 1) \cdot \#\Big(\pi_0\big(\mathcal T(v^g, \om)\big)\Big)
 \geq \; \textup{rank} \big(H^{\mathbf\D}_{j}(DM)\big). 
 \qquad \diamondsuit
\]
\end{itemize}
\end{corollary}

\noindent {\bf Example 3.2.} Consider a collection $\{\Sigma'_k\}_{k \in [1, s]}$ of  closed surfaces of genera $\mathsf g_k \geq 2$. Let $N'_k = \Sigma'_k \times S^2$. We denote by $N'$ the connected sum of all $N'_k$'s. Let $M'$ be the compact $4$-dimensional manifold, obtained from $N'$ by removing  a number of smooth $4$-balls $D^4$ and solid tori of the form $T^2 \times D^2$ and $S^1 \times D^3$, residing in $N'$. We assume that these domains do not intersect the surfaces $\{\Sigma'_k \times pt_k\}$, where the points $pt_k \in S^2$. Finally, let $M$ be any smooth compact $4$-dimensional manifold which is homotopy equivalent to $M'$ and such that $\d M = \d M'$. Let $H: M \to M'$ be this homotopy equivalence. We denote by $\Sigma_k$ be the $H$-preimage of $\Sigma'_k$. We may assume that $H$ is transversal to $\coprod_{k=1}^s (\Sigma'_k\times pt_k) \subset M'$. 

Then the hyperbolicity of the $\Sigma'_k$'s implies that 
$$\textup{rank} \big(H^{\mathbf\D}_{2}(SM)\big) \geq s \text{\; and \;} \textup{rank} \big(H^{\mathbf\D}_{2}(DM)\big) \geq 2s.$$

By Theorem \ref{th11.17}, for any traversally generic metric $g$ on $M$, the number of $5$-dimensional connected components in the stratification $\{SM^\circ(v^g, \om)\}_{\om}$ of $SM^\circ$ is $s$ at least. These are connected components of the strata, indexed by $\om$'s with the property $|\om|' = 2$. Such $\om$'s belong to the list $\{(1221), (13), (3,1)\}$. 
\smallskip   

Let us rephrase these conclusions in terms of $M$. For any traversally generic metric $g$ on $M$, there exists at least $s$  distinct four-dimensional families of geodesics $\g$ in $M$, such that either $\g$ is quadratically tangent to $\d M$ at a pair of distinct points, or $\g$ has a single tangency to $\d M$ of order $3$. Each family is continuously parametrized by an open smooth $4$-manifold. Different families do not share the same geodesics. \smallskip

Next, let us interpret the claims of Theorem \ref{th11.16} in this setting. Take, for example, the  $2$-cycle $h = \sum_{k=1}^s [\Sigma_k]$.  Its simplicial norm $\| h\|_{\mathsf \D} = \sum_{k=1}^s (2\mathsf g_k -2)$. 

For an universal constant $\mu \geq 1$ and any traversally generic metric $g$ on $M$,  the number of $5$-dimensional connected components in the stratification $\{SM^\circ(v^g, \om)\}_{\om}$ of $SM$, where $\om = (1221), (13), (3,1)$, is greater than or equal to $\mu^{-1} \cdot \sum_{k=1}^s (2 \mathsf g_k -2)$. Although estimating $\mu$ from above may be challenging, at least we know the rate of growth of the number of $5$-dimensional connected components in the $\mathcal S_{v^g}(SM^\circ)$-stratification, as $s \to \infty$ or as individual genus $\mathsf g_k \to \infty$. 
\hfill $\diamondsuit$ \smallskip

Revisiting Definition \ref{def2.3}, the formulas from Corollary \ref{cor11.A} have an instant implication. 

\begin{corollary}\label{cor11.B} Under the hypotheses of Theorem \ref{th11.17}, the non-triviality of the groups $H^{\mathbf\D}_{j}(M)$, where $j \in [1, n-1]$, and/or $H^{\mathbf\D}_{j}(DM)$ where $j \in [1, n]$, represents an obstruction to the existence of a globally $j$-convex and traversally generic metric on $M$. 
\hfill $\diamondsuit$
\end{corollary}

The assumption that a metric $g$ on a given manifold $M$ is traversally generic, perhaps, could be relaxed. To extend all the results of this paper to the broader class of boundary generic and traversing geodesic flows, we need to consider only such geodesic flows that match a \emph{finite list} of a priori fixed semi-local models (in the spirit of (\ref{eq2.4})) of the vicinity of every $v^g$-trajectory $\g$. Perhaps, these models could be determined not only by the combinatorics of tangency (like $\om$) of $\g$ to $\d(SM)$, but also by some continuous parameters.

\end{document}